\documentclass[oneside,english]{amsart}
\usepackage[T1]{fontenc}
\usepackage[latin9]{inputenc}
\usepackage{geometry}
\usepackage{graphicx}
\usepackage{psfrag}
\usepackage{amsthm}
\usepackage{amstext}
\usepackage{amssymb}
\usepackage{esint}

\makeatletter
\numberwithin{equation}{section}
\numberwithin{figure}{section}
\theoremstyle{plain}
\newtheorem{thm}{\protect\theoremname}
  \theoremstyle{remark}
  \newtheorem{rem}[thm]{\protect\remarkname}
  \theoremstyle{plain}
  \newtheorem{lem}[thm]{\protect\lemmaname}
  \theoremstyle{plain}
  \newtheorem{cor}[thm]{\protect\corollaryname}
  \theoremstyle{definition}
  \newtheorem{defn}[thm]{\protect\definitionname}
  \theoremstyle{plain}
  \newtheorem{prop}[thm]{\protect\propositionname}
  \theoremstyle{definition}
  \newtheorem{example}[thm]{\protect\examplename}

\numberwithin{thm}{section}

\providecommand{\constname}{Construction}

\usepackage{babel}

\global\long\def\Bz#1{B\left({x}^{*}\left(0\right),#1\right)}

\global\long\def\R{\mathbb{R}}

\global\long\def\xs#1{x^*(#1)}
\global\long\def\ys#1{y^*(#1)}
\global\long\def\n#1{\mathbf{N}_{#1}}

\global\long\def\xdt{\frac{d}{dt}x^*(t)}

\global\long\def\xt{x^*(t)}

\DeclareMathOperator{\Area}{Area} 
\DeclareMathOperator{\Vol}{Vol} 
\DeclareMathOperator{\Int}{Int} 
\DeclareMathOperator{\Ext}{Ext} 
\renewcommand{\setminus}{\smallsetminus}

\def\crn#1#2{{\vcenter{\vbox{
  \hbox{\kern#2pt \vrule width.#2pt height#1pt}
    \hrule height.#2pt}}}}
\newcommand{\intprod}{\mathchoice\crn54\crn54\crn{3.75}3\crn{2.5}2} 
\newcommand{\into}{\mathbin{\intprod}}
\renewcommand{\phi}{\varphi}

\usepackage{babel}\providecommand{\corollaryname}{Corollary}
  \providecommand{\definitionname}{Definition}
  \providecommand{\examplename}{Example}
  \providecommand{\lemmaname}{Lemma}
  \providecommand{\propositionname}{Proposition}
  \providecommand{\remarkname}{Remark}
\providecommand{\theoremname}{Theorem}

\makeatother

\usepackage{babel}
  \providecommand{\corollaryname}{Corollary}
  \providecommand{\definitionname}{Definition}
  \providecommand{\examplename}{Example}
  \providecommand{\lemmaname}{Lemma}
  \providecommand{\propositionname}{Proposition}
  \providecommand{\remarkname}{Remark}
\providecommand{\theoremname}{Theorem}

\begin{document}

\title{A Note on Flux Integrals over Smooth Regular Domains}
\begin{abstract}
We provide new bounds on a flux integral over the portion of the
boundary of one regular domain contained inside a second regular domain,
based on properties of the second domain rather than the first one.
This bound is amenable to numerical computation of a flux through
the boundary of a domain, for example, 
when there is a large variation in the normal vector near a point.
We present applications of this result to occupational measures and two-dimensional differential equations, including a 
new proof that all minimal invariant sets in the plane are trivial.

\end{abstract}

\author{Ido Bright}

\author{John M. Lee}

\address{Department of Applied Mathematics, University of Washington.}

\address{Department of Mathematics, University of Washington.}

\maketitle

\section{Introduction}

\global\long\def\R{\mathbb{R}}

\global\long\def\S{\mathbb{S}}

\global\long\def\n{\mathbf{n}}

A \emph{regular domain} in $\R^{d}$ is a closed, embedded $d$-dimensional
smooth submanifold with boundary, such as a closed ball or a closed
half-space. (Throughout this paper, \emph{smooth} means
infinitely differentiable.)  If $D\subset\R^{d}$ is a regular domain, 
its interior $\mathring D$ is an open subset of $\R^d$, and 
its boundary $\partial D$
is a closed, embedded, codimension-$1$ smooth submanifold (without
boundary) which is the common topological boundary of the open sets
$\mathring D$ and $\R^{d}\setminus D$.  For this reason, the boundary
of a regular domain is often called a \emph{space-separating hypersurface}.
The Jordan--Brouwer separation theorem (see, for example \cite[p.~89]{guillemin2010differential})
shows that if $S\subset\R^{d}$ is any compact, connected, embedded
hypersurface, then the complement of $S$ has two connected components,
one bounded (the \emph{interior} of $S$) and another unbounded (the
\emph{exterior} of $S$), with $S$ as their common boundary; thus
$S\cup\Int S$ and $S\cup\Ext S$ are both regular domains. But in
general, the boundary of a regular domain need not be connected  (for example, an annulus in
the plane).

Surface integrals computing the flux through boundaries of regular
domains are ubiquitous in physics and engineering. We present two
bounds for surface integrals on a portion of the boundary of one
domain contained inside a second domain. 
The results are presented for regular domains in Euclidean space for simplicity, but
Theorems \ref{thm:Stokes-Approx-1}  and \ref{thm:Stokes-Approx-2} 
extend to regular domains in Riemannian manifolds. See Theorem \ref{thm:riemannian}.
For more details about the
notation in these theorems, see Section \ref{sec:Transversal-Intersection-of}.

\global\long\def\theenumi{\alph{enumi}}
 \global\long\def\labelenumi{\textup{(}\theenumi\textup{)}}
 \begin{thm} \label{thm:Stokes-Approx-1} 
Suppose $D_{1},D_{2}\subset\R^{d}$
are regular domains, such that $D_{1}\cap D_{2}$ is compact and $D_2$ has finite volume
and surface area. Suppose $f$
is a smooth vector field defined on a neighborhood of $D_2$ such
that $|f|$ and $|\nabla\cdot f|$ are bounded. Then the absolute
value of the flux of $f$ across the portion of $\partial D_{1}$
inside $D_{2}$ satisfies 
\begin{equation}
\left|\int_{\partial D_{1}\cap D_{2}}f\cdot\n_{\partial D_{1}}\, dA\right|\le\Area\left(\partial D_{2}\right)\left\Vert f\right\Vert _{\infty}+\Vol(D_{2})\left\Vert \nabla\cdot f\right\Vert _{\infty}.\label{eq:flux ineq}
\end{equation}
 \end{thm}

When the vector field is divergence-free, we have the following much better bound.

 \begin{thm} \label{thm:Stokes-Approx-2} 
Suppose $D_{1},D_{2}\subset\R^{d}$
are regular domains with compact intersection and finite surface areas, and 
$f$
is a smooth bounded vector field on $\R^d$ satisfying
$\nabla\cdot f\equiv0$.  Then
\begin{equation}
\left|\int_{\partial D_{1}\cap D_{2}}f\cdot\n_{\partial D_{1}}\, dA\right|\le\frac{1}{2}\Area\left(\partial D_{2}\right)\left\Vert f\right\Vert _{\infty}.\label{eq:flux ineq div-free}
\end{equation}
 \end{thm}

A surprising corollary to Theorem \ref{thm:Stokes-Approx-2} bounds
the integral of the normal vector of the portion of a hypersurface
contained inside a second regular domain. 
\begin{cor} \label{cor:main result}
Suppose $D_{1},D_{2}\subset\R^{d}$ are regular domains with compact intersection and finite surface areas.
The following inequality holds: 
\begin{equation}
\left|\int_{\partial D_{1}\cap D_{2}}\n_{\partial D_{1}}\, dA\right|\le\frac{1}{2}\Area(\partial D_{2}).\label{main inequality}
\end{equation}
 \end{cor}

When $D_{2}$ is convex we have the following alternative bound, which
is an improvement in some cases.

\begin{thm} \label{thm:convex-theorem} Suppose $D_{1},D_{2}\subset\R^{d}$
are regular domains. If $D_{2}$ is compact and convex with diameter
$\delta$, then 
\begin{equation}
\left|\int_{\partial D_{1}\cap D_{2}}\n_{\partial D_{1}}\, dA\right|\le\frac{1}{2}\Vol\left(B^{d-1}(\delta/2)\right),\label{convex inequality-1}
\end{equation}
 where $B^{d-1}(\delta/2)$ denotes the ball in $\R^{d-1}$ of radius
$\delta/2$. \end{thm}

The significance of these  results is that, although the integration
is with respect to the portion of $\partial D_{1}$ inside $D_{2}$,
which might have arbitrarily large surface area (see Fig.\ \ref{fig:1})),%
\begin{figure}
\psfrag{D1}{$D_1$}
\psfrag{D2}{$D_2$}
\includegraphics{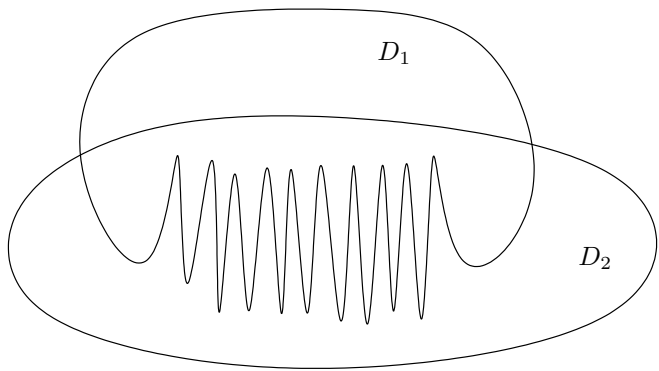}
\caption{The setup for Theorems \ref{thm:Stokes-Approx-1} and \ref{thm:Stokes-Approx-2}}
\label{fig:1}
\end{figure}
 the bound depends
only on $D_{2}$. This is due to the cancellations of the normal vector
that occur in hypersurfaces that bound regular domains, and would
not hold for images of general immersions of codimension 1 (see Example
\ref{Ex: need_stokes-1}).

Theorem \ref{thm:Stokes-Approx-1} is applicable to
the numerical computation of the flux on the surface of a regular
domain when there is a large variation of the normal vector near a point, resulting in
a large surface area contained in a region
of small volume. Indeed, the flux over the problematic part can be
estimated by finding a domain containing it, avoiding direct computation.
We provide an application of Corollary \ref{cor:main result} in Section
\ref{sec:surface-limit}, for limits of sequences of regular domains
with surface area increasing without bound; there we use the bound
to show that in the limit, the average velocity, say in a ball, is
zero. Such a result is applied in the case $d=2$, in Artstein and
Bright \cite{artstein2010periodic}, to obtain a new Poincar\'e--Bendixson
type result for planar infinite-horizon optimal control.

Corollary \ref{cor:main result} generalizes a previous result, for
$d=2$, established by Artstein and Bright \cite{artstein2010periodic,artsteinvelocity}.
This topological result has proved fruitful in applications, providing
new Poincar\'e--Bendixson type results, in an optimal-control setting
\cite{artstein2010periodic,bright2012reduction}, and in the context
of dynamics with no differentiability assumptions by Bright \cite{bright2012reduction}.
The proofs of the planar result in \cite{artstein2010periodic,artsteinvelocity}
employ a dynamical argument, which is similar to the one used in the
textbook proof of the Poincaré--Bendixson theorem. In this paper, we
generalize the results to boundaries of open sets, restricting ourselves
in this presentation to regular domains; however the results hold
for more general sets and vector fields. The results in their fullest
generality for non-smooth domains and fluxes are presented in Bright
and Torres \cite{BrightTorres}. 

\begin{rem} The requirement that
$D_{1}\cap D_{2}$ be compact is essential, as it implies
that $\partial D_{1}\cap D_{2}$ is compact, so that the integrals
in \eqref{eq:flux ineq}--\eqref{main inequality}  are finite.
\end{rem}

\begin{rem} Theorem \ref{thm:Stokes-Approx-1} can be extended, by
replacing the smooth vector field $f$ with a smooth matrix-valued
function $\Pi$, using the induced norm. \end{rem}

\begin{rem}
For simplicity, 
Theorem \ref{thm:Stokes-Approx-2}
is stated under the assumption that $f$ is defined on all of $\R^d$;
but as the proof will show, if $D_2$ has finite volume it is only necessary
that $f$ be defined on some neighborhood of $D_2$.
\end{rem}

The structure of this paper is as follows. The following section presents
notations and lemmas used in the paper. In Section \ref{sec:Proof Thm}
we prove Theorems \ref{thm:Stokes-Approx-1} and  \ref{thm:Stokes-Approx-2},
and describe how our results extend to regular domains in a Riemannian manifold. 
In Section \ref{sec:thorem normal_int}
we prove Corollary \ref{cor:main result} and Theorem
\ref{thm:convex-theorem},
and also provide examples showing the tightness of the bound.
In the last section we provide three applications of Corollary \ref{cor:main result}:
an application to limits of sequences of regular domains; an extension when $d=2$;
and a simplified proof of a theorem on invariant sets for dynamical systems.

\section{\label{sec:Transversal-Intersection-of} Notations \& Lemmas}

Throughout this paper, we denote the characteristic function of a
set $A\subset\R^{d}$ by $\chi_{A}$. The $d$-dimensional volume
is denoted by $\Vol(A)$, and the $(d-1)$-dimensional surface area
of its boundary by $\Area(\partial A)$. Given two submanifolds $S_{1},S_{2}\subset\R^{d}$,
the notation $S_{1}\pitchfork S_{2}$ means that $S_{1}$ and $S_{2}$
intersect transversally. The Euclidean norm on $\R^{d}$ is denoted
by $\left|\cdot\right|$, and the supremum norm on functions by $\left\Vert \cdot\right\Vert _{\infty}$.
The divergence of a smooth vector field $f=\left(f^{1},f^{2},\dots,f^{d}\right)$
at the point $x=\left(x^{1},x^{2},\dots,x^{d}\right)\in\R^{d}$ is
denoted by 
\[
\nabla\cdot f(x)=\frac{\partial}{\partial x^{1}}f^{1}(x)+\frac{\partial}{\partial x^{2}}f^{2}(x)+\cdots+\frac{\partial}{\partial x^{d}}f^{d}(x).
\]

The following is a simple lemma we need for the proof of the main
theorems.

\begin{lem} \label{lem:measure lemma AB} Suppose $(X,\mu)$ is a
measure space, $U,V\subset X$, and $U$ has finite measure.
For every real-valued
function $f\in L^{\infty}\left(X\right)$, 
\[
\left|\int_{U\setminus V}f(x)\mu\left(dx\right)\right|\le\frac{1}{2}\left(\mu\left(U\right)\left\Vert f\right\Vert _{\infty}+\left|\int_{U}f(x)\mu\left(dx\right)\right|\right),
\]
 and 
\[
\left|\int_{U\cap V}f(x)\mu\left(dx\right)\right|\le\frac{1}{2}\left(\mu\left(U\right)\left\Vert f\right\Vert _{\infty}+\left|\int_{U}f(x)\mu\left(dx\right)\right|\right).
\]
\end{lem} 

\begin{proof} 
The first inequality follows from the triangle inequality:
\begin{align*}
\int_{U}\left|f(x)\right|\mu\left(dx\right) & \ge\left|\int_{U\setminus V}f(x)\mu\left(dx\right)\right|+\left|\int_{U\cap V}f(x)\mu\left(dx\right)\right|\\
 & =\left|\int_{U\setminus V}f(x)\mu\left(dx\right)\right|+\left|\int_{U\setminus V}f(x)\mu\left(dx\right)-\int_{U}f(x)\mu\left(dx\right)\right|\\
 & \ge2\left|\int_{U\setminus V}f(x)\mu\left(dx\right)\right|-\left|\int_{U}f(x)\mu\left(dx\right)\right|.
\end{align*}
The second inequality follows by replacing 
$V$ with $X\setminus V$.
\end{proof} 

The proofs of the main theorems are based on the divergence theorem for certain domains in $\R^d$.
Let us say a \emph{regular domain with corners} in $\R^d$ is a closed
subset $D\subset\R^d$ such that for each point $p\in D$,
there exist an open set $U\subset\R^d$ containing $p$
and a smooth coordinate chart $\phi\colon U\to \R^d$ such that 
$\varphi(U\cap D)$ is the intersection of 
$\phi(U)$ with $\overline{\R}{}_{+}^{d}=\{x\in\R^{d}\mid x^{1}\ge0,\dots,x^{d}\ge0\}$.
Some typical examples are closed simplices and rectangular solids.
Every regular domain is a regular domain with corners,
and a regular domain with corners is a $d$-dimensional smooth manifold with corners in the sense 
defined in \cite{lee2013introduction}.

Here is the version of the divergence theorem we will use.

\begin{lem}\label{lemma:divergence-with-corners}
Suppose $D\subset\R^d$ is a regular domain with corners, which has finite
volume and surface area.  
If $f$ is a smooth vector field defined on $D$ such that both $|f|$ and $|\nabla\cdot f|$
are bounded, then 
\begin{equation*}
\int_{\partial D} f \cdot \n_{\partial D}\, dA = \int_D \nabla\cdot f\, dV.
\end{equation*}
\end{lem}

\begin{proof}
If $D$ is compact, or more generally if $f$ is compactly supported,
this follows immediately from Stokes's theorem
applied to the $(d-1)$-form $f\into (dx^1\wedge\dots\wedge dx^d)$, where
``$\into$'' denotes interior multiplication.
(For Stokes's theorem
on manifolds with corners see, for example, \cite[Thm. 16.25, p. 419]{lee2013introduction}.)
In the general case, we argue as follows.  
Let $\phi\colon[0,\infty)\to[0,1]$
be a smooth function that is equal to $1$ on $[0,\tfrac{1}{2}]$
and supported in $[0,1]$, and for each $r>0$ let $\phi_{r}(x)=\phi\big(|x|^{2}/r^{2}\big)$.
Then the vector field $\phi_{r}f$ is compactly supported, so the
divergence theorem implies 
\begin{equation}
\int_{\partial D}\phi_{r}f\cdot\n_{\partial D}\, dA=\int_{D}\nabla\cdot(\phi_{r}f)\, dV.\label{eq:cutoff-divergence}
\end{equation}
As $r\to\infty$, the integral on the left-hand side of \eqref{eq:cutoff-divergence}
converges to $\int_{\partial D} f\cdot\n_{\partial D}\, dA$
by the dominated convergence theorem. On the other hand, for each
$r>0$, 
\[
\big|\nabla\cdot(\phi_{r}f)(x)\big|
=\left|\phi_r(x)\nabla\cdot f(x) +\frac{2}{r^{2}}\sum_{i=1}^{d}\phi'\left(\frac{|x|^{2}}{r^{2}}\right)x^{i}f^{i}(x)\right|
\le\|\nabla\cdot f\|_\infty +\frac{2}{r}\|\phi'\|_\infty\|f\|_{\infty},
\]
because $|x|\le r$ on the support of $\phi'\big(|x|^2/r^2\big)$. Since $\nabla\cdot(\phi_{r}f)$
converges pointwise to $\nabla\cdot f$ and $D$ has finite volume, it follows
from the dominated convergence theorem that the right-hand side of
\eqref{eq:cutoff-divergence} converges to $\int_{D}\nabla\cdot f\, dV$.
\end{proof}

The next proposition is used in the proof of the main
theorems.

\begin{prop} \label{prop:epsilon-approx} Suppose $D_{1}$ and $D_{2}$
are regular domains in $\R^{d}$, with $D_{1}\cap D_{2}$ compact
and with $D_{2}$ of finite volume and surface area.  Suppose further that $f$ is a 
smooth bounded vector field defined on a neighborhood of $D_2$.  
There exists a sequence of regular domains $D_{2,i}$ such that $\partial D_{2,i}$ is transverse to $\partial D_{1}$,
and 
the following limits hold
as $i\to\infty$:
\begin{enumerate}
\item $\Vol\big(D_{2,i}\big)\to \Vol\left(D_{2}\right)$; 
\item $\Area\big(\partial D_{2,i}\big)\to \Area\left(\partial D_{2}\right)$;
\item $\int_{\partial D_{2,i}}f\cdot\n_{\partial D_{2,i}}\, dA\to \int_{\partial D_{2}}f\cdot\n_{\partial D_{2}}\, dA$.\label{eq:surf_int cnv}
\end{enumerate}
The domains can be chosen so that $D_{2,i}$ is either a decreasing sequence of domains
whose intersection is $D_2$, or an increasing sequence of domains whose union is $\mathring D_2$.  
 \end{prop}

\begin{proof} 
As a smooth embedded hypersurface, $\partial D_{2}$
has a tubular neighborhood $N$, and there exists a smooth embedding
$E\colon\partial D_{2}\times\left(-\delta,\delta\right)\rightarrow N$
such that $E\left(\cdot,0\right)$ is the identity on $\partial D_{2}$.
It can be chosen such that 
$E\left(\partial D_{2}\times\left(0,\delta\right)\right)\cap D_{2}=\emptyset$
and 
$E\left(\partial D_{2}\times (-\delta,0]\right)\subset D_{2}$.


Let $W\subset\R^{d}$ be a precompact neighborhood of $D_{1}\cap D_{2}$
contained in the set on which $f$ is defined,
and let $\varphi\colon\R^{d}\to[0,1]$ be a smooth compactly supported
function that is equal to $1$ on $\overline{W}$. For each $\eta$
such that $\delta>\eta>0$, define 
\begin{align*}
V_{\eta} & =\{E(x,s):0\le s\le\eta\varphi(x)\},\\
D_2^{\eta} & =D_{2}\cup V_{\eta}.
\end{align*}
(See Fig.\ \ref{fig:2}.)%
\begin{figure}
\psfrag{D2}{$D_2$}
\psfrag{Veta}{$V_\eta$}
\includegraphics{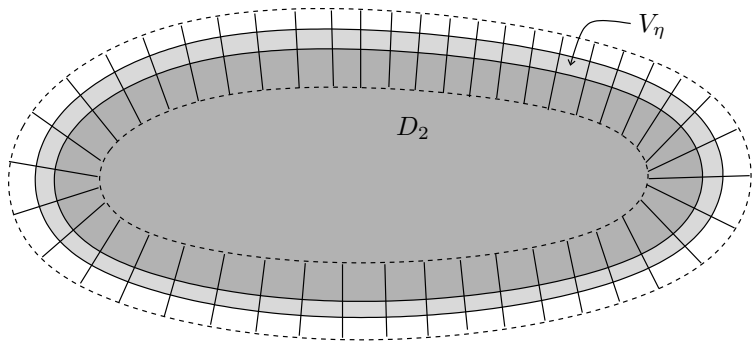}
\caption{Defining a domain $D^\eta_2$ containing $D_2$}
\label{fig:2}
\end{figure}
 Then $D_2^{\eta}$ is a regular domain containing $D_{2}$, which agrees
with $D_{2}$ outside the support of $\varphi$. Its boundary $\partial D_2^{\eta}$
is the image of the embedding $\iota_{\eta}\colon \partial D_{2}\to\R^{d}$
given by $\iota_{\eta}(x)=E(x,\eta\varphi(x))$, which is equal to
the inclusion map $\partial D_{2}\hookrightarrow\R^{d}$ outside $\text{supp}\ \varphi$.
The map $E$
has full rank in $(\partial D_{2}\cap W)\times(-\delta,\delta)$, and $\varphi\equiv 1$
there, so by the
parametric transversality theorem (see, for example, \cite[Thm. 6.35, p. 145]{lee2013introduction}),
$\partial D_2^{\eta}$ is transverse to $\partial D_{1}$ for almost
every $\eta\in\left(-\delta,\delta\right)$. 

Now let $\eta_i$ be a sequence of positive numbers that decreases to zero, 
chosen so that $\partial D_2^{\eta_i}$ is transverse to $\partial D_{1}$ for each $i$,
and set $D_{2,i} = D_2^{\eta_i}$.  Then 
$D_{2,i}$ decreases to $D_2$ and $\Vol(D_{2,i})$ decreases to $\Vol(D_2)$.  
Moreover, because the embeddings
$\iota_{\eta_i}$ converge uniformly with all derivatives to the inclusion
map $\partial D_{2}\hookrightarrow\R^{d}$,
the surface area of $\partial D_{2,i}$ converges to that of $\partial D_{2}$.
Furthermore, the function $\n_{\partial D_{2,i}}\circ\iota_{\eta_i}\colon\partial D_{2}\to\R^{d}$
converges to $\n_{\partial D_{2}}$. Combining these two arguments,
we conclude that \eqref{eq:surf_int cnv} is satisfied.

To obtain a sequence of domains that increase to $\mathring D_2$, 
we proceed instead as follows.  For each $\eta$ such that
$-\delta<\eta<0$, define 
\begin{align*}
V_{\eta} & =\{E(x,s):\eta\varphi(x)< s \le 0\},\\
D_2^{\eta} & =D_{2}\setminus V_{\eta}.
\end{align*}
In this case, we can choose a sequence of negative numbers $\eta_i$ increasing
to zero such that $\partial D_2^{\eta_i}$ is transverse to $\partial D_2$,
and the rest of the proof proceeds as before.
\end{proof}

\section{\label{sec:Proof Thm}Proof of Theorems \ref{thm:Stokes-Approx-1} and  \ref{thm:Stokes-Approx-2}}

In this section, we prove Theorems \ref{thm:Stokes-Approx-1} and  \ref{thm:Stokes-Approx-2}. We start
with a more general result that implies both theorems;
first, we prove it when the boundaries of the domains intersect transversally,
then, employing an approximation argument, we prove the general case.

\begin{thm} \label{thm: extended_stokes}Suppose $D_{1}$ and $D_{2}$
are two regular domains in $\R^{d}$, such that $D_{1}\cap D_{2}$ is compact
and $D_{2}$ has finite volume and surface area. Let $f$ be a smooth
vector field defined on a neighborhood of $D_{2}$, such that both
$|f|$ and $|\nabla\cdot f|$ are bounded. The absolute value of the
flux of $f$ across the portion of $\partial D_{1}$ inside $D_{2}$
satisfies the following bound: 
\begin{multline}
\left|\int_{\partial D_{1}\cap D_{2}}f\cdot\n_{\partial D_{1}}\, dA\right|\\
\le\frac{1}{2}\left(\Area\left(\partial D_{2}\right)\left\Vert f\right\Vert _{\infty}
+\left|\int_{\partial D_{2}}f\cdot\n_{\partial D_{2}}\, dA\right|+\Vol(D_{2})\left\Vert \nabla\cdot f\right\Vert _{\infty}
+\left|\int_{D_{2}}\nabla\cdot f\, dV\right|\right).\label{eq:Stokes_gen_bound}
\end{multline}
The same estimate holds when $\partial D_{1}\cap D_{2}$ is replaced by $\partial D_{1}\cap\mathring D_2$ on the left-hand side.  
 \end{thm}

\begin{prop} \label{prop:trans-inter}Theorem \ref{thm: extended_stokes}
holds when $\partial D_{1}\pitchfork\partial D_{2}$.\end{prop}

\begin{proof} Note that $\partial(D_{1}\cap D_{2})$ is compact, and
\[
\partial\big(D_{1}\cap D_{2}\big)=\left(\partial D_{1}\cap D_{2}\right)\cup\left(D_{1}\cap\partial D_{2}\right).
\]
 Adding and subtracting $\int_{\partial D_{2}\cap D_{1}}f\cdot\n_{\partial D_{2}}\, dA$,
we obtain 
\begin{equation}
\begin{aligned}\int_{\partial D_{1}\cap D_{2}}f\cdot\n_{\partial D_{1}}\, dA & =\int_{\partial D_{1}\cap D_{2}}f\cdot\n_{\partial D_{1}}\, dA+\int_{\partial D_{2}\cap D_{1}}f\cdot\n_{\partial D_{2}}\, dA-\int_{\partial D_{2}\cap D_{1}}f\cdot\n_{\partial D_{2}}\, dA\\
 & =\int_{\partial\left(D_{1}\cap D_{2}\right)}f\cdot\n_{\partial\left(D_{1}\cap D_{2}\right)}\, dA-\int_{\partial D_{2}\cap D_{1}}f\cdot\n_{\partial D_{2}}\, dA,
\end{aligned}
\label{eq:normals}
\end{equation}
 since $\partial D_{1}\cap\partial D_{2}$ is a smooth $(d-2)$-dimensional
submanifold and thus has zero $(d-1)$-dimensional area.

The assumption $\partial D_{1}\pitchfork\partial D_{2}$ implies that
$D_{1}\cap D_{2}$ is a smooth manifold with corners. To see this,
we just need to show that each point is contained in the domain of
an appropriate smooth coordinate chart. 
For points not in $\partial D_{1}\cap\partial D_{2}$, this follows
easily from the fact that $D_{1}$ and $D_{2}$ are regular domains.
If $x\in\partial D_{1}\cap\partial D_{2}$, we can
find a local defining function $u^{1}$ for $D_{1}$, such that $D_{1}$
is locally given by the equation $u^{1}\ge0$; and similarly we can
find a local defining function $u^{2}$ for $D_{2}$. The assumption
$\partial D_{1}\pitchfork\partial D_{2}$ ensures that $du^{1}$ and
$du^{2}$ are linearly independent at $x$. Thus we can find smooth
functions $u^{3},\dots,u^{d}$ such that $(u^{1},\dots,u^{d})$ form
the required local coordinates in a neighborhood of $x$.

Applying the divergence theorem, we get 
\[
\int_{\partial D_{1}\cap D_{2}}f\cdot\n_{\partial D_{1}}\, dA=\int_{D_{2}\cap D_{1}}\nabla\cdot f\, dV-\int_{\partial D_{2}\cap D_{1}}f\cdot\n_{\partial D_{2}}\, dA.
\]

Applying Lemma \ref{lem:measure lemma AB} to both terms on the right
hand side completes the proof for $\partial D_{1}\cap D_2$.
The result for $\partial D_{1}\cap\mathring D_2$ is immediate  in this case, because $\partial D_{1}\cap\partial D_{2}$
has zero surface area.
\end{proof}

\begin{proof}[Proof of Theorem \ref{thm: extended_stokes}] Let $D_{2,i}$ be a sequence
of regular domains decreasing to $D_2$ and satisfying the conclusions of Proposition \ref{prop:epsilon-approx}.
By Proposition \ref{prop:trans-inter}, for every $i$ we have that
$\left|\int_{\partial D_{1}\cap D_{2,i}}f\cdot\n_{\partial D_{1}}\, dA\right|$
is bounded by 
\[
\frac{1}{2}\left(\Area\left(\partial D_{2,i}\right)\left\Vert f\right\Vert _{\infty}
+\left|\int_{\partial D_{2,i}}f\cdot\n_{\partial D_{2,i}}\, dA\right|
+\Vol(D_{2,i})\left\Vert \nabla\cdot f\right\Vert _{\infty}
+\left|\int_{D_{2,i}}\nabla\cdot f\, dV\right|\right).
\]
Proposition \ref{prop:epsilon-approx} shows that
the first three terms above converge to the first three terms 
on the right-hand side of \eqref{eq:Stokes_gen_bound}.
To complete the proof, we use the facts that 
the sets $D_{2,i}$ decrease
to $D_{2}$ and the compact sets $\partial D_{1}\cap D_{2,i}$
decrease to $\partial D_{1}\cap D_{2}$ as $i$ goes
to infinity, and thus the Lebesgue dominated convergence theorem yields
\[
\lim_{i\rightarrow\infty}\left|\int_{D_{2,i}}\nabla\cdot f\, dV\right|=\left|\int_{D_{2}}\nabla\cdot f\, dV\right|
\]
 and 
\[
\lim_{i\rightarrow\infty}\left|\int_{\partial D_{1}\cap D_{2,i}}f\cdot\n_{\partial D_{1}}\, dA\right|=\left|\int_{\partial D_{1}\cap D_{2}}f\cdot\n_{\partial D_{1}}\, dA\right|.
\]
 This completes the proof for $\partial D_{1}\cap D_{2}$.

To prove the estimate for $\partial D_{1}\cap \mathring D_{2}$, we use the same argument,
but with $D_{2,i}$ chosen to increase to $\mathring D_2$.  
Because $\partial D_2$ has $d$-dimensional measure zero, we have $\int_{\mathring D_{2}}\nabla\cdot f\, dV = 
\int_{D_{2}}\nabla\cdot f\, dV$, and the result follows.
\end{proof}

\begin{proof}[Proof of Theorem \ref{thm:Stokes-Approx-1}] 
Inequality
\eqref{eq:flux ineq} follows immediately from \eqref{eq:Stokes_gen_bound}
and obvious estimates for the integrals. 
\end{proof}

\begin{proof}[Proof of Theorem \ref{thm:Stokes-Approx-2}] 
We first assume that 
$\Vol\left(D_{2}\right)<\infty$, so that 
\eqref{eq:Stokes_gen_bound} holds. 
In this case, the last two terms
in \eqref{eq:Stokes_gen_bound} are zero because $\nabla\cdot f=0$, and 
the second term is zero by the divergence theorem.

Now consider the case in which $D_2$ has infinite volume.  
Let $D_2'$ denote the closure of ${\R^d \setminus D_2}$, which is a regular domain with interior
$\mathring D_2' = \R^d \setminus D_2$.
Because $\Area(\partial D_2') = \Area(\partial D_2)<\infty$, the isoperimetric
inequality (see \cite{de1953definizione}) implies that $D_2'$ has finite
volume.  If $D_1$ also has finite volume, the divergence theorem gives
\[
\int_{\partial D_{1}\cap D_{2}}f\cdot\n_{\partial D_{1}}\, dA+\int_{\partial D_{1}\cap\mathring D_2'}f\cdot\n_{\partial D_{1}}\, dA=\int_{\partial D_{1}}f\cdot\n_{\partial D_{1}}\, dA=\int_{D_{1}}\nabla\cdot f\, dV=0,
\]
and 
\eqref{eq:flux ineq div-free} follows from 
Theorem \ref{thm: extended_stokes} applied to the second term on the left-hand
side above.  
On the other hand, if $\Vol(D_1)=\infty$, we let $D_1'$ be the closure of $\R^d\setminus D_1$
(which has finite volume), and apply the above argument with $D_1'$ in place of $D_1$.
\end{proof}

To conclude this section, we explain what modifications need to be
made to Theorems \ref{thm:Stokes-Approx-1} and \ref{thm:Stokes-Approx-2} and their proofs to adapt
them to the case of regular domains in Riemannian manifolds.

Suppose $M$ is a $d$-dimensional smooth Riemannian manifold with
Riemannian metric $g$ and volume density $dV_g$. (If $M$ is oriented,
$dV_g$ can be interpreted as a differential $d$-form; but otherwise
it needs to be interpreted as a density. See \cite[pp. 427--434]{lee2013introduction}
for basic properties of densities.) A \emph{regular domain} $D\subset M$
is defined just as in the case $M=\R^{d}$. If $D\subset M$ is a
regular domain, it has a uniquely defined outward unit normal vector
field $\n_{\partial D}$. For any such domain, we let 
$\tilde g$ denote the induced Riemannian metric on $\partial D$, and let
$dA_{\tilde g}$ denote its volume density.

For any smooth vector field $f$ defined on an open subset of $M$,
the \emph{divergence} of $f$, denoted by $\nabla\cdot f$, is defined
as follows. If $M$ is oriented, then $\nabla\cdot f$ is the unique
vector field that satisfies $(\nabla\cdot f)dV_g=d(f\into dV_g)$. On a nonorientable
manifold, we define it locally by choosing an orientation and using
the same formula; because $\nabla\cdot f$ is unchanged when the orientation
is reversed, it is globally defined. The divergence theorem then holds
in exactly the same form for smooth $d$-dimensional submanifolds
with corners in $M$. Moreover, any compact smooth embedded hypersurface
in $M$ has a tubular neighborhood in $M$. (See Bredon \cite[p. 100, Thm. 11.14]{bredon1993topology}
for a proof.). Although the proof there is for manifolds embedded
in Euclidean space, it follows from the Whitney embedding theorem
that it applies to all smooth manifolds.)

Using these facts, the proof of the following theorem is carried out
exactly like the proofs of Theorems \ref{thm:Stokes-Approx-1} and  \ref{thm:Stokes-Approx-2}. To avoid
complications, we restrict to the case in which $D_{2}$ is compact.

\begin{thm}\label{thm:riemannian} 
If $D_{1}$ and $D_{2}$ are regular domains in a Riemannian
manifold $\left(M,g\right)$ with $D_{2}$ compact, and $f$ is a smooth
vector field defined on a neighborhood of $D_2$,
then the conclusions of Theorems \ref{thm:Stokes-Approx-1}
and \ref{thm:Stokes-Approx-2}
hold, namely,
\[
\left|
\int_{\partial D_{1}\cap D_{2}}\left\langle f,\mathbf{n}_{\partial D_{1}}\right\rangle _{g}dA_{\tilde{g}}
\right|
\le\Area\left(\partial D_{2}\right)\left\Vert f\right\Vert _{\infty}+\Vol(D_{2})\left\Vert \nabla\cdot f\right\Vert _{\infty},
\]
and if $\nabla\cdot f\equiv 0$, 
\[
\left|
\int_{\partial D_{1}\cap D_{2}}\left\langle f,\mathbf{n}_{\partial D_{1}}\right\rangle _{g}dA_{\tilde{g}}
\right|
\le\frac{1}{2}\Area\left(\partial D_{2}\right)\left\Vert f\right\Vert _{\infty}.
\]
\end{thm}

\section{\label{sec:thorem normal_int}Bounding integrals of normal fields}

In this section, we prove Corollary \ref{cor:main result} and Theorem
\ref{thm:convex-theorem}. We also provide examples on the tightness
of the bound.

\begin{proof}[Proof of Corollary \ref{cor:main result}] Let
$v=\int_{\partial D_{1}\cap D_{2}}\n_{\partial D_{1}}\, dA$. If $\left|v\right|=0$
there is nothing to prove, so we assume that $\left|v\right|>0$,
and let $f\colon\R^{d}\to\R^{d}$ be the constant vector field $f\equiv v/\left|v\right|$.
Clearly, $|v|=v\cdot v/|v|=\int_{\partial D_{1}\cap D_{2}}f\cdot\n_{\partial D_{1}}\, dA$.
Now, since $\nabla\cdot f\equiv0$ and $\|f\|_{\infty}=1$, the proof follows from Theorem
\ref{thm:Stokes-Approx-2}.\end{proof}

To prove Theorem \ref{thm:convex-theorem}, we begin with a lemma.

\begin{lem}\label{lemma:vdotn} Suppose $D\subset\R^{d}$ is a compact
convex regular domain with diameter $\delta$, and $C$ is any measurable
subset of $\partial D$. Then for any unit vector $v\in\R^{d}$, we
have 
\begin{equation}
\int_{C}v\cdot\n_{\partial D}\, dA\le\frac{1}{2}\Vol\left(B^{d-1}(\delta/2)\right).\label{eq:vdotn}
\end{equation}
 \end{lem}

\begin{proof} First consider the case $v=e_{d}=(0,\dots,0,1)$. After
applying a translation, we can assume that $D$ is contained in the
set where $x^{d}\ge0$. Its boundary is the union of the three subsets
$\partial D_{+}$, $\partial D_{0}$, and $\partial D_{-}$, defined
as the subsets of $\partial D$ where $v\cdot\n_{\partial D}$ is
positive, zero, or negative, respectively.

Now, let $A$ be the following subset of $\R^{d}$: 
\[
A=\{(x^{1},\dots,x^{d-1},tx^{d})\mid(x^{1},\dots,x^{d})\in D,\ 0\le t\le1\}.
\]
 Then $A$ is a compact convex set, and its boundary is the union
of the three subsets $\partial A_{+}$, $\partial A_{0}$, and $\partial A_{-}$,
defined in the same way as above. 
(See Fig.\ \ref{fig:3}.)%
\begin{figure}
\psfrag{D}{$D$}
\psfrag{A}{$A$}
\psfrag{v}{$v$}
\psfrag{xd}{$x^d$}
\psfrag{x1,...,xd-1}{$x^1,\dots,x^{d-1}$}
\psfrag{dD0}{$\partial D_0$}
\psfrag{dA0}{$\partial A_0$}
\psfrag{dD-}{$\partial D_-$}
\psfrag{dA-}{$\partial A_-$}
\psfrag{dA+=dD+}{$\partial A_+=\partial D+$}
\includegraphics{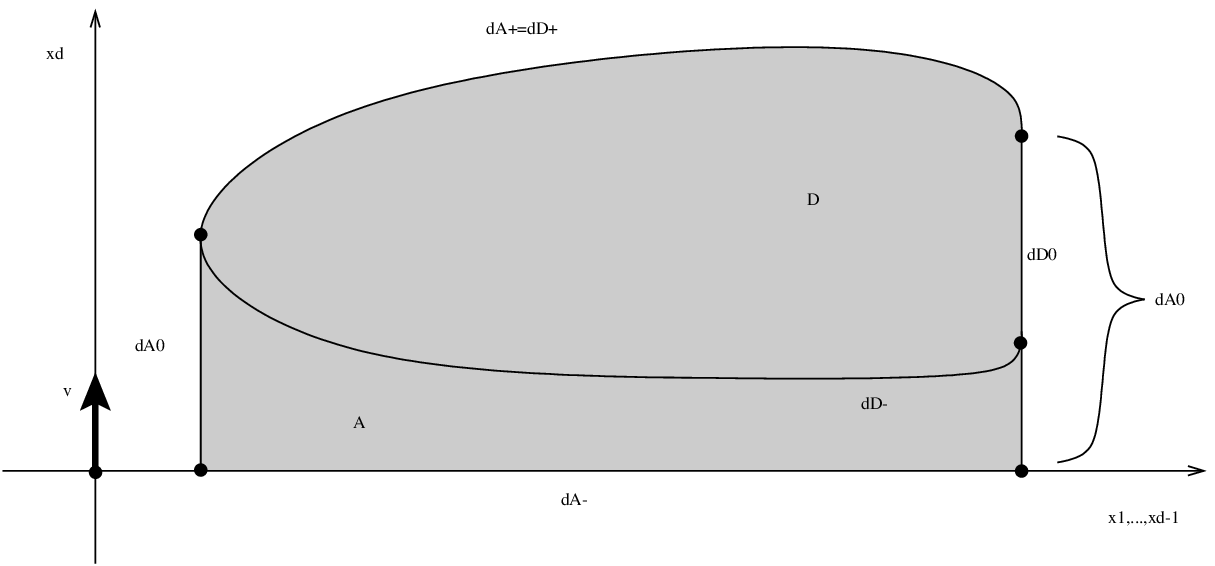}
\caption{Proof of Lemma \ref{lemma:vdotn}}
\label{fig:3}
\end{figure}

The fact that $D$ is convex ensures
that $\partial A_{+}=\partial D_{+}$, $\partial A_{0}\supset\partial D_{0}$,
and $\partial A_{-}$ is contained in the hyperplane where $x^{d}=0$.
Moreover, $A$ is a $C^{1}$ manifold with corners. (Its boundary
might not be smooth at points where $\partial A_{0}$ meets $\overline{\partial A}_{+}$,
but it is at least $C^{1}$ there.)

Using the fact that $v\cdot\n_{\partial D}<0$ on $\partial D_{-}$
and $v\cdot\n_{\partial D}=0$ on $\partial D_{0}$, we compute 
\begin{align*}
\int_{C}v\cdot\n_{\partial D}\, dA & =\int_{C\cap\partial D_{+}}v\cdot\n_{\partial D}\, dA+\int_{C\cap\partial D_{0}}v\cdot\n_{\partial D}\, dA+\int_{C\cap\partial D_{-}}v\cdot\n_{\partial D}\, dA\\
 & \le\int_{C\cap\partial D_{+}}v\cdot\n_{\partial D}\, dA\le\int_{\partial D_{+}}v\cdot\n_{\partial D}\, dA\\
 & =\int_{\partial A_{+}}v\cdot\n_{\partial A}\, dA=-\int_{\partial A_{-}}v\cdot\n_{\partial A}\, dA,
\end{align*}
 where in the last line we have used the divergence theorem for the
vector field $f\equiv v$ and the fact that $v\cdot\n_{\partial A}=0$ on
$\partial A_{0}$. Since $\n_{\partial A}=-v$ on $\partial A_{-}$,
the last integral is equal to the area of $\partial A_{-}$. Since
$\partial A_{-}$ is contained in a $(d-1)$-dimensional ball of radius
$\delta/2$, the result follows.

Finally, for the case of a general unit vector $v$, we just apply
a rotation to $D$ and apply the above argument. \end{proof}

\begin{proof}[Proof of Theorem \ref{thm:convex-theorem}] Let
$D_{1}$ and $D_{2}$ be as in the statement of the theorem. If $\int_{\partial D_{1}\cap D_{2}}\n_{\partial D_{1}}\, dA=0$,
there is nothing to prove, so assume the integral is nonzero, and
let $v$ be the unit vector in the direction of $\int_{\partial D_{1}\cap D_{2}}\n_{\partial D_{1}}\, dA$.
Then 
\[
\left|\int_{\partial D_{1}\cap D_{2}}\n_{\partial D_{1}}\, dA\right|=v\cdot\int_{\partial D_{1}\cap D_{2}}\n_{\partial D_{1}}\, dA=\int_{\partial D_{1}\cap D_{2}}v\cdot\n_{\partial D_{1}}\, dA,
\]
 and the result follows from Lemma \ref{lemma:vdotn}. \end{proof}

The following examples demonstrate the tightness of the bound for
non-convex sets, as well as the necessity of the condition that the
hypersurface be the boundary of a regular domain. \begin{example}
\label{Ex: need_stokes-1}The main theorem explicitly uses the divergence
theorem, which is applied to space-separating hypersurfaces. In fact,
the bounds do not apply for images of general smooth immersions. To
construct a counterexample in the plane (i.e., for $d=2$), start
with a smooth Jordan curve in the plane, then cover it $m$ times,
with small perturbations, making the integral on the left-hand side
of \eqref{main inequality} roughly $m$ times as large, while the
right-hand side is fixed because it depends only on $\partial D_{2}$.
Clearly, whenever the left-hand side of \eqref{main inequality} is
not zero, we can choose $m$ large enough that the inequality does
not hold. \end{example}

\begin{example} \label{Ex: tight bounds-1}To see that the bound
obtained in Theorem \ref{cor:main result} is tight, and cannot be
replaced by a bound based only on the diameter of $D_{2}$ when $D_{2}$
is not convex, we consider comb-shaped subsets of $\R^{d}$, for $d\ge2$,
generated in the following manner: Fix $n>2$, and let $D_{n}$ be
a closed non-smooth comb-shaped set defined as the union of the following
rectangles: 
\begin{align*}
R_{i,n} & =\big\{ x=\left(x^{1},\dots,x^{d}\right)\in\left[0,1\right]^{d}\ \big|\ i/n\le x^{2}\le i/n+1/n^{2}\big\}\mbox{ for }i=0,1,2,\dots,n-1;\\
R_{n,n} & =\big\{ x=\left(x^{1},\dots,x^{d}\right)\in\left[0,1\right]^{d}\ \big|\ 0\le x^{1}\le1/n^{2}\big\}.
\end{align*}
 Applying a small perturbation we then smooth its corners, and set
$D_{1,n}$ accordingly. Let $D_{2,n}$ be the translation of $D_{1,n}$
by the vector $\left(1/\left(2n^{2}\right),1/\left(2n^{2}\right),0,\dots,0\right)\in\R^{d}$.
(See Fig.\ \ref{fig:4}.)%
\begin{figure}
\psfrag{D14}{$D_{1,4}$}
\psfrag{D24}{$D_{2,4}$}
\includegraphics{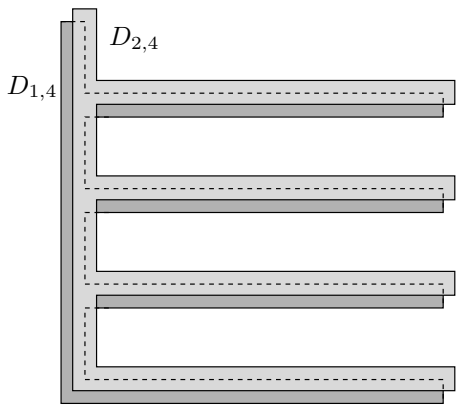}
\caption{The domains of Example \ref{Ex: tight bounds-1} in the case $n=4$ (before smoothing)}
\label{fig:4}
\end{figure}
By our construction, the surface area of each set $\partial D_{1,n}$
or $\partial D_{2,n}$ is roughly $2n+2$, 
and the area of the portion where the normal vector of
$\partial D_{1,n}$ 
is parallel to the $x^{2}$-axis is roughly $n/(n+2)$, approaching $1$ when $n$ is large.
Notice that by the choice of $D_{2,n}$, when we integrate the normal
vector in the portion of $\partial D_{1,n}$ inside $D_{2,n}$ we
capture only the part pointing in the positive direction of the $x^{2}$-axis.
This shows that the integral of the normal vector has magnitude of
roughly $n$, approaching half the surface area when we take $n$
to infinity. \end{example}

\section{\label{sec:surface-limit}Applications: Limits of Hypersurfaces \&
Planar Results}


In this section we provide two applications of Theorem \ref{cor:main result},
extending previous planar results in \cite{artstein2010periodic,artsteinvelocity}.
The first is for limits of regular domains whose surface areas increase
without bound. The second is an application in the planar case.

Corollary \ref{cor:main result} bounds the normal vector of the boundary
of a regular domain in a second regular domain, 	 by the surface area
of the boundary of the second domain, and completely disregarding
the surface area of the original hypersurface. This is now applied
to surfaces with increasing surface area, establishing a new result
on the limit.

In what follows, we denote by $\S^{d-1}\subset\R^{d}$ the unit $(d-1)$-sphere.
For every hypersurface we define a corresponding probability measure
using the following notation:

\begin{defn} Suppose $S\subset\R^{d}$ is a smooth hypersurface endowed
with a unit normal vector field $\n_{S}$. We define the \emph{empirical
measure} $\mu\in P\left(\R^{d}\times\S^{d-1}\right)$ corresponding
to $S$ by 
\[
\mu\left(U\times V\right)=\frac{1}{\Area\left(S\right)}\int_{S\cap U}\chi_{V}\big(\n_{S}\big)\, dA,
\]
 for all open sets $U\subset\R^{d}$ and $V\subset\S^{d-1}$. \end{defn}

A useful property of empirical measures is the following fact: if
$f\colon{\R^{d}\times\S^{d-1}}\rightarrow\R$ is continuous, we have
\[
\frac{1}{\Area\left(S\right)}\int_{S}f\big(x,\n_{S}\big)\, dA=\int_{\R^{d}\times\S^{d-1}}f\left(x,n\right)\mu\left(dx,dn\right).
\]

We endow the set of probability measures $P\left(\R^{d}\times\S^{d-1}\right)$
with the weak topology, namely, a sequence of measures $\mu_{1},\mu_{2},\ldots\in P\left(\R^{d}\times\S^{d-1}\right)$
converges to a measure $\mu_{0}\in P\left(\R^{d}\times\S^{d-1}\right)$
if for every bounded continuous function $g\left(x,n\right)$, 
\[
\int_{\R^{d}\times\S^{d-1}}g\left(x,n\right)\mu_{0}\left(dx,dn\right)=\lim_{i\rightarrow\infty}\int_{\R^{d}\times\S^{d-1}}g\left(x,n\right)\mu_{i}\left(dx,dn\right).
\]

Another tool we need for the next theorem is disintegration of measures.
Given a probability measure $\mu\in P\left(\R^{d}\times\S^{d-1}\right)$,
we define its marginal measure, $p\left(dx\right)$, as the projection
on $\R^{d}$, namely, $p\left(A\right)=\mu\left(A\times\S^{d-1}\right)$
for every measurable set $A\subset\R^{d}$. Also, we denote the measure
valued function $\mu^{x}(dn)$, the disintegration of $\mu$ with
respect to $p$, for $p$-almost every $x$. With this notation, for
every pair of measurable sets $U\subset\R^{d}$ and $V\subset\S^{d-1}$,
we have that $\mu\left(U\times V\right)=\int_{U}\mu^{x}(V)p\left(dx\right)$.

We now state the main result regarding the limits of regular domains.

\begin{thm} \label{thm:surface limit} Let $D_{1},D_{2},\ldots\subset\R^{d}$
be a sequence of compact regular domains, such that the surface areas
of their boundaries increases	 to infinity. If the empirical measures
$\mu_{1},\mu_{2},\dots$, corresponding to the sequence $\partial D_{1},\partial D_{2},\ldots$,
converge weakly to $\mu_{0}$, then 
\[
h(x)=\int_{\S^{d-1}}n\, \mu_{0}^{x}(x)\left(dn\right)=0
\]
 for $p_{0}$-almost every $x$, where $\mu_{0}(dx,dn)=p_{0}(dx)\mu_{0}^{x}(dn)$
is the disintegration of $\mu_{0}$ with respect to its projection,
$p_{0}$.\end{thm}

\begin{proof} Let $B=B\left(x,r\right)\subset\R^{d}$ be a ball centered
at $x$ with radius $r>0$. By the definition of the empirical measures
and by Corollary \ref{cor:main result}, 
\[
\left|\int_{B\times\S^{d-1}}n\, d\mu_{i}\left(dx,dn\right)\right|=\left|\frac{1}{\Area\left(\partial D_{i}\right)}\int_{\partial D_{i}\cap B}\n_{\partial D_{i}}\, dA\right|\le\frac{\Area\left(\partial B\right)}{2\Area\left(\partial D_{i}\right)}.
\]
 The weak convergence of measures and the dominated convergence theorem
imply that 
\[
\left|\int_{B\times\S^{d-1}}n\, d\mu_{0}\left(dx,dn\right)\right|=\lim_{i\rightarrow\infty}\left|\int_{B\times\S^{d-1}}n\, d\mu_{i}\left(dx,dn\right)\right|\le\lim_{i\rightarrow\infty}\frac{\Area\left(\partial B\right)}{2\Area\left(\partial D_{i}\right)}=0,
\]
 for a set of values of $r>0$ of full measure for which $\mu_{i}\left(\partial B\left(x,r\right)\times\S^{d-1}\right)=0$,
for all $i=0,1,2,\dots$. Using the disintegration notation we obtain
that 
\[
\left|\int_{B\times\S^{d-1}}n\, d\mu_{0}\left(dx,dn\right)\right|=\left|\int_{B}\left(\int_{\S^{d-1}}n\,\mu_{0}^{x}\left(dn\right)\right)p_{0}\left(dx\right)\right|=\int_{B}h(x)p_{0}\left(dx\right)=0
\]
 for almost every ball $B$. If the measure $p_{0}\left(dx\right)$
is Lebesgue measure, by the Lebesgue differentiation theorem
we have $h(x)=0$ almost everywhere. The Lebesgue-Besicovitch
differentiation theorem extends this result to Radon measures (see,
for example, Evans and Gariepy  \cite[page 43]{evans1991measure}).\end{proof}

\begin{rem} Theorem \ref{thm:surface limit} requires the convergence
of the empirical measures. When the domains in the sequence are contained
in some compact set $K$, the compactness of the space $K\times\S^{d-1}$
implies the compactness of $P\left(K\times\S^{d-1}\right)$; and therefore,
the existence of a converging subsequence \cite[p. 72]{billingsley2009convergence}.\end{rem}

In two dimensions, our result extends as follows.

\begin{cor} \label{cor:2d_cor}Suppose $x_{1}\colon[0,\tau_{1}]\to\R^{2}$
is a parametrized smooth Jordan curve and $D_{2}\subset\R^{2}$ is a regular domain.
If the length of $\partial D_{2}$ is $L_{2}$, then 
\[
\left|\int_{0}^{\tau_{1}}\chi_{D_{2}}\left(x_{1}(t)\right)\frac{d}{dt}x_{1}(t)dt\right|\le\frac{L_{2}}{2}.
\]
 \end{cor}

\begin{proof} Let $T_{1}$ and $N_{1}$ be the unit tangent and normal
vectors of $x_{1}$. Using the arc-length param\-e\-tri\-za\-tion, we have
that 
\[
\left|\int_{0}^{\tau_{1}}\chi_{D_{2}}\left(x_{1}(t)\right)\frac{d}{dt}x_{1}(t)dt\right|=\left|\int_{0}^{L_{1}}\chi_{D_{2}}\left(x_{1}(s)\right)T_{1}(s)ds\right|,
\]
 where $L_{1}$ is the length of $x_{1}$. Expressing the tangent
vector in terms of the normal vector, we reduce the previous expression
to 
\[
\left|\int_{0}^{L_{1}}\chi_{D_{2}}\left(x_{1}(s)\right)\left[\begin{array}{cc}
0 & -1\\
1 & 0
\end{array}\right]N_{1}(s)ds\right|=\left|\int_{0}^{L_{1}}\chi_{D_{2}}\left(x_{1}(s)\right)N_{1}(s)ds\right|,
\]
 as the rotation matrix is orthogonal. Applying Theorem \ref{cor:main result}
completes the proof.\end{proof}

For our final application, we consider an ordinary differential equation
in the plane defined by 
\begin{equation}
\frac{d {x}}{dt}=f( {x}),\label{eq:ode}
\end{equation}
 where $f\colon\R^{2}\rightarrow\R^{2}$ is a vector field (generally
assumed at least Lipschitz continuous). An \emph{invariant set} for
$f$ is a subset of $\R^{2}$ that is invariant under the forward
flow of $f$ and a \emph{minimal set} is a nonempty closed invariant
set that is minimal with respect to inclusions. A \emph{trivial minimal
set} is a set that is the image of either a stationary solution or
a periodic solution.

We present a new short proof of the following well-known result.

\begin{thm}\label{thm:minimal-set} Suppose $f$ is a smooth vector
field on $\R^{2}$. Then every minimal set for $f$ is trivial. \end{thm}

The textbook proof of this theorem (see Verhulst \cite{verhulst1996nonlinear})
relies on the Poincaré\textendash{}Bendixson theorem, and employs
dynamical arguments. Here we present a simpler proof based on the divergence
theorem, and specifically on Corollary \ref{cor:main result}. Note
that the divergence theorem was used by Bendixson in the proof of
the Bendixson criterion, which verifies that no periodic solutions
exist. 

Our proof uses the following well-known lemmas. 

\begin{lem} \label{lem:recurense} Suppose $\Omega\subset\R^{2}$
is a minimal set for \eqref{eq:ode} and $ {x}^{*}\colon[0,\infty)\to\R^{2}$
is a solution to \eqref{eq:ode} with trajectory contained in $\Omega$.
For every $ {y_{0}}\in\Omega$, $s\in[0,\infty)$, and $\delta>0$,
there exists $t>s$ such that $\left|\xs t- {y_{0}}\right|<\delta$.\end{lem}

\begin{proof} Suppose the lemma does not hold for some $ {y_{0}}$,
$s$, and $\delta$. Then the curve $\ys t=\xs{s+t}$ is a solution to
(\ref{eq:ode}) with trajectory contained in $\Omega\setminus B( {y_{0}},\delta)$
for a suitable $\delta>0$, in contradiction to the minimality of
$\Omega$. \end{proof}

The next lemma follows easily from Sard's theorem. \begin{lem} \label{lem:Sard}Suppose
$I\subset\R$ is a bounded interval and $g:I\rightarrow\R$ is smooth.
Then for almost every $r\in\R$, the set $g^{-1}(r)=\left\{ t\in I\mid g\left(t\right)=r\right\} $
is finite. \end{lem}

\begin{proof}[Proof of Theorem \ref{thm:minimal-set}] 
Clearly,
$\Omega$ is a singleton if and only if it contains a point $ {y}\in\Omega$
such that $f\left( {y}\right)= {0}$, so we may assume
henceforth that $f$ does not vanish in $\Omega$ and $\Omega$ contains
more than one point. Choose $D>0$ such that $\Omega\setminus B(\xs0,3D)\neq\emptyset$.
We construct sequences of real numbers $\{\delta_{i}\}$ and $\{t_{i}\}$,
and a sequence of simple closed curves $\{\gamma_{i}\}$, as follows.
Set $\delta_{0}=D$, and let $t_{0}$ be the first time where $ {x}^{*}$
meets $\partial\Bz{\delta_{0}}$. For $i=1,2,\dots$ do the following: 
\begin{enumerate}
\item Choose $\delta_{i}<\delta_{i-1}/2$ small enough that $\left| {x}^{*}(0)- {x}^{*}(t)\right|>\delta_{i}$
for all $t\in\left[t_{0},t_{i-1}\right]$. 
\item Let $t_{i}$ be the first time after $t_{0}$ where the curve $ {x}^{*}$
meets $\partial B( {x}^{*}(0),\delta_{i})$. (Here we use Lemma
\ref{lem:recurense}.) 
\item Starting from $\xs{t_{i}}$, follow the line connecting it to $ $$\xs0$,
until that line first meets a point in $ {x}^{*}\big({\left[0,t_{0}\right]}\big)$.
Let $\xs{s_{i}}$ be this point. 
(See Fig.\ \ref{fig:5}.)%
\begin{figure}
\psfrag{xt0}{$x^*(t_0)$}
\psfrag{xt1}{$x^*(t_1)$}
\psfrag{xs1}{$x^*(s_1)$}
\psfrag{xt0}{$x^*(t_0)$}
\psfrag{x0}{$x^*(0)$}
\includegraphics{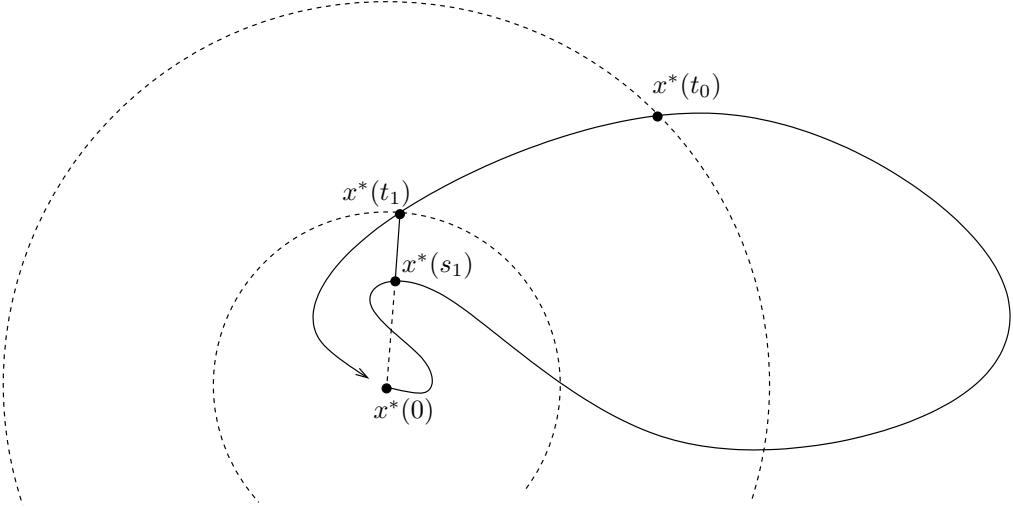}
\caption{Proof of Theorem \ref{thm:minimal-set}}
\label{fig:5}
\end{figure}
\item Let $\gamma_{i}$ be the parametrized piecewise smooth curve obtained by following
the curve $ {x}^{*}$ in the interval $\left[s_{i},t_{i}\right]$,
and then the line connecting its endpoints with unit speed. 
\end{enumerate}
Note that $t_{i}$ is an increasing sequence and that the uniqueness
of the solution with respect to the initial condition implies that
every $\gamma_{i}$ is a Jordan curve. Suppose first that the sequence
$\{t_{i}\}$ is bounded above. Then $t_{i}\rightarrow t^{*}\in\R^{+}$
and $ {x}^{*}(t_i)\rightarrow {x}^{*}(t^{*})$. According to our construction,
$\left|\xs0-\xs{t_{i}}\right|=\delta_{i}<2^{-i}D$ for every $i$.
Hence, by continuity $\xs{t^{*}}=\xs0$, and $ {x}^{*}$ is
periodic. By the minimality of $\Omega$, the image of $ {x}^{*}$
is $\Omega$.

The only remaining possibility is $t_{i}\nearrow\infty$. Fix $ {y}_{0}\in\Omega$
such that $\left| {y}_{0}-\xs0\right|>2D$. By Lemma \ref{lem:Sard},
there exists arbitrarily small $r_{0}<D$ such that the set $\left\{ t\in[0,s)\mid\left| {x}^{*}(t)- {y_{0}}\right|=r_{0}\right\} $
is finite for every $s>0$. (This follows from the fact that $g\left(t\right)=\left| {x}^{*}(t)- {y_{0}}\right|^{2}$
is a smooth function of $t$.) Note that this implies that the portion of $\gamma_{i}$ in
$B_0=B\left(y_0,r_0\right)$ is part of the trajectory $ {x}^{*}$, and that for every
$i$ the Jordan curve $\gamma_{i}$ intersects $\partial B_0$ at a
finite number of points.

For every $i$, we let $D_{i}$ denote the domain consisting of the
Jordan curve $\gamma_{i}$ together with its interior. Although $D_{i}$
is not a regular domain, it is a regular domain with two corner points,
which are outside of $\overline{B}_0$, and it is easy to see that Corollary
\ref{cor:main result} can be applied to $\partial D_{i}\cap B_0$.
Thus by Corollary \ref{cor:2d_cor},
\[
\left|\int_{\left\{ t\le t_{i}\mid\xs t\in B_0\right\} }\xdt dt\right|=\left|\int_{\left\{ t\le t_{i}\mid\xs t\in B_0\right\} }f\left(\xt\right)dt\right|\le \pi r_{0}.
\]
 Because $\Omega$ is minimal, Lemma \ref{lem:recurense} implies
that the set $\left\{ t\mid\xs t\in B_0\right\} $ has infinite measure.
This implies that ${0}$ is contained in the convex hull of
the set $\left\{ f\left( {y}\right)\mid {y}\in\overline{B}_0\right\} $.
The radius $r_{0}$ can be chosen arbitrary small; therefore, the
continuity of $f$ implies that $f\left( {y}_{0}\right)={0}$,
in contradiction. \end{proof}

\textbf{Acknowledgement}: The first author wishes to thank Vered Rom-Kedar
for suggesting the idea of studying the integral of the normal vector,
and Monica Torres, for valuable remarks.
 \bibliographystyle{plain}
\bibliography{bib}

\begin{thebibliography}{10}

\bibitem{artstein2010periodic}
Zvi Artstein and Ido Bright.
\newblock Periodic optimization suffices for infinite horizon planar optimal
  control.
\newblock {\em SIAM J. Control Optim.}, 48(8):4963--4986, 2010.

\bibitem{artsteinvelocity}
Zvi Artstein and Ido Bright.
\newblock On the velocity of planar trajectories.
\newblock {\em NoDEA Nonlinear Differential Equations Appl.}, 20(2):177--185,
  2013.

\bibitem{billingsley2009convergence}
Patrick Billingsley.
\newblock {\em Convergence of probability measures}.
\newblock John Wiley \& Sons Inc., New York, second edition, 1999.

\bibitem{bredon1993topology}
Glen~E. Bredon.
\newblock {\em Topology and geometry}.
\newblock Springer-Verlag, New York, 1993.

\bibitem{bright2012reduction}
Ido Bright.
\newblock A reduction of topological infinite-horizon optimization to periodic
  optimization in a class of compact 2-manifolds.
\newblock {\em J. Math. Anal. Appl.}, 394(1):84--101, 2012.

\bibitem{BrightTorres}
Ido Bright and M.~Torres.
\newblock Estimates on the flux of a vector field over boundaries of sets of
  finite perimeter.
\newblock Conference video, \emph{Nonlinear Conservation Laws and Related
  Models}, Banff, June, 2013,
  www.birs.ca/events/2013/5-day-workshops/13w5061/videos.

\bibitem{de1953definizione}
Ennio De~Giorgi.
\newblock Definizione ed espressione analitica del perimetro di un insieme.
\newblock {\em Atti Accad. Naz. Lincei Cl. Sci. Fis. Mat. Nat}, 8(14):390--393,
  1953.

\bibitem{evans1991measure}
Lawrence~C. Evans and Ronald~F. Gariepy.
\newblock {\em Measure theory and fine properties of functions}.
\newblock CRC Press, Boca Raton, FL, 1992.

\bibitem{guillemin2010differential}
Victor Guillemin and Alan Pollack.
\newblock {\em Differential topology}.
\newblock AMS Chelsea Publishing, Providence, RI, 2010.
\newblock Reprint of the 1974 original.

\bibitem{lee2013introduction}
John~M. Lee.
\newblock {\em Introduction to smooth manifolds}.
\newblock Springer, New York, second edition, 2013.

\bibitem{verhulst1996nonlinear}
Ferdinand Verhulst.
\newblock {\em Nonlinear differential equations and dynamical systems}.
\newblock Springer Verlag, 1996.

\end{thebibliography}

\end{document}